\documentclass[3p,times]{elsarticle}
\usepackage[colorlinks=true,urlcolor=blue,pdfstartview=FitH, linkcolor=blue,citecolor=blue,bookmarksopen,bookmarksnumbered={true}]{hyperref}
\usepackage{enumerate}
\usepackage{amsmath}
\usepackage{amsfonts}
\usepackage{amssymb}
\usepackage{amsthm}
\biboptions{comma,square,numbers,sort&compress}
\journal{...}
\newtheorem{thm}{Theorem}
\newtheorem{cor}{Corollary}
\newtheorem{lem}{Lemma}
\newtheorem{prop}{Proposition}
\theoremstyle{definition}
\newtheorem{defn}{Definition}
\theoremstyle{remark}
\newtheorem{rem}{Remark}
\newtheorem{example}{Example}

\allowdisplaybreaks

\begin{document}
	\begin{frontmatter}
		\title{Relation-theoretic metrical coincidence and common fixed point theorems under nonlinear contractions}
		\author[mad]{Md Ahmadullah\corref{cor1}}
		\cortext[cor1]{Corresponding author}
		\ead{ahmadullah2201@gmail.com}
		\author[mid]{Mohammad Imdad}
		\ead{mhimdad@gmail.com}
		\author[maf]{Mohammad Arif}
		\ead{mohdarif154c@gmail.com}
		\address[mad]{Department of Mathematics, Aligarh Muslim University, Aligarh,-202002, U.P., India.}
		\address[mid]{Department of Mathematics, Aligarh Muslim University, Aligarh,-202002, U.P., India.}
		\address[maf]{Department of Mathematics, Aligarh Muslim University, Aligarh,-202002, U.P., India.}
		
		\begin{abstract}
			In this paper, we prove coincidence and common fixed points results under nonlinear contractions on a metric space equipped with an arbitrary binary relation. Our results extend, generalize, modify and unify several known results especially those are contained in Berzig [J. Fixed Point Theory Appl. 12, 221-238 (2012))]  and Alam and Imdad [To appear in Filomat (arXiv:1603.09159 (2016))]. Interestingly, a corollary to one of our main results under symmetric closure of a binary relation remains a sharpened version of a theorem due to Berzig. Finally, we use examples to highlight the accomplished improvements in the results of this paper.
		\end{abstract}
	\end{frontmatter}

		\vspace{.3cm}
		\section{Introduction and Preliminaries} Banach contraction principle (see \cite{Bnch1922}) continues to be one of the most inspiring and core result of metric fixed point theory which also has various applications in classical functional analysis besides several other domains especially in mathematical economics and psychology. In the course of last several years, numerous authors have extended this result by weakening the contraction conditions besides enlarging the class of underlying metric space. In recent years such type of results are also established employing order-theoretic notions. Historically speaking, the idea of order-theoretic fixed points was initiated by Turinici \cite{Turinicid1986} in 1986. In 2004, Ran and Reurings \cite{RanR2004} formulated a relatively more natural order-theoretic version of classical Banach contraction principle.
		 The existing literature contains several relation-theoretic results on fixed, coincidence and common fixed point (e.g., partial order: Ran and Reurings \cite{RanR2004} and Nieto and Rodr\'{i}guez-L\'{o}pez \cite{NietoL2005}, tolerance:  Turinici \cite{Turinici2011,Turinici2012}, strict order: Ghods et al. \cite{GGGH2012}, transitive: Ben-El-Mechaiekh \cite{BenEM2015}, preorder: Turinici \cite{Turinici2013} etc). Berzig \cite{Berzig2012} established the common fixed point theorem for nonlinear contraction under symmetric closure of a arbitrary relation. Most recently, Alam and Imdad \cite{Alamimdad2} proved a relation-theoretic version of Banach contraction principle employing amorphous relation which in turn unify the several well known relevant order-theoretic fixed point theorems. Moreover, for further  details one can consults \cite{AhmadJI,AhmadIA,Alamimdad,Alamimdad2,Alamimdad3,BenEM2015,Berzig2012,Ciric1974,Jach,RanR2004,NietoL2005,SametT2012,Turinici2011,Turinici2012}.
		
		\vspace{.3cm}
		Our aim in this work is to proved some coincidence and common fixed point theorems for nonlinear contraction on metric space endowed with amorphous relation. The results proved herein generalize and unify main results  of Berzig \cite{Berzig2012}, Alam and Imdad \cite{Alamimdad2} and several others. To demonstrate the validity of the hypotheses and degree of generality of our results, we  also furnish some examples.
		
		\section{Preliminaries}
		For the sake of simplicity to have possibly self-contained presentation, we require some basic definitions, lemmas and propositions  for our subsequent discussion.
		\begin{defn} \cite{Jungck1976,Jungck1996} Let $(f,g)$ be a pair of self-mappings defined on a non-empty set  $X$. Then
			\begin{enumerate}[$(i)$]
				\item a point $u\in X$ is said to be a coincidence point of the pair $(f,g)$ if
				$fu=gu,$
				\item a point $v\in X$ is said to be a point of coincidence of the pair $(f,g)$ if there exists $u\in X$ such that $v=fu=gu,$
				\item a coincidence point $u\in X$ of the pair $(f,g)$ is said to be a common fixed point if $u=fu=gu,$
				\item  a pair $(f,g)$ is called commuting if $f(gu)=g(fu), \forall ~u\in X$.
			\end{enumerate}
		\end{defn}
		
		\begin{defn} \cite{Jungck1986, Sastry2000, Sessa1982} Let $(f,g)$ be a pair of self-mappings defined on a metric space $(X,d).$ Then
			\begin{enumerate}[$(i)$]
				\item $(f,g)$ is said to be weakly commuting if for all $u\in X$,
				~~	$d(f(gu),g(fu))\leq d(fu,gu),$
				\item  $(f,g)$ is said to be compatible if
				$\lim_{n\to \infty}d(f(gu_n),g(fu_n))=0$
				whenever $\{u_n\}\subset X$ is a sequence such that $\lim_{n\to \infty} gu_n=\lim_{n\to \infty}fu_n,$
				
				\item  $f$ is said to be a $g$-continuous at $u\in X$ if ${gu_n}\stackrel{d}{\longrightarrow}gu,$ for all sequence $\{u_n\}\subset X$, we have ${fu_n}\stackrel{d}{\longrightarrow} fu.$
				Moreover, $f$ is said to be a $g$-continuous if it is continuous at every point of $X.$
			\end{enumerate}
		\end{defn}
		
		\begin{defn} \cite{Lips1964} A subset $\mathcal{R}$ of
			$X\times X$ is called a binary relation on X. We say that
			``$u$ relates $v$ under $\mathcal{R}$" if and only if $(u,v)\in
			\mathcal{R}$.
		\end{defn}
		
		Throughout this paper, $\mathcal{R}$ stands for a `non-empty binary relation' ($i.e.,\mathcal{R} \ne \emptyset$) instead of `binary relation'
		while $\mathbb{N}_{0},$ $\mathbb{Q}$ and $\mathbb{Q}^{c}$ stand the set of whole numbers ($\mathbb{N}_{0}= \mathbb{N}\cup \{0\}$), the set of rational numbers and the set of irrational numbers respectively.
		
		\begin{defn} \cite{Maddux2006} A binary relation $\mathcal{R}$ defined on a non-empty set $X$ is called complete if every pair of elements of $X$ are comparable under that relation $i.e.,$ for all
			$u, v$ in $X,$ either $(u,v)\in \mathcal{R}$ or $(v,u)\in \mathcal{R}$  which is denoted by $[u,v]\in \mathcal{R}$.
		\end{defn}
		
		\begin{prop} \cite{Alamimdad} Let $\mathcal{R}$ be a binary relation defined on a non-empty set $X$. Then
			$(u,v)\in\mathcal{R}^s$ if and only if  $[u,v]\in\mathcal{R}.$
		\end{prop}
		
		\begin{defn} \cite{Alamimdad} \label{3.5} Let $f$ be a self-mapping defined on a non-empty set $X$. Then a binary relation
			$\mathcal{R}$ on $X$ is called $f$-closed if ${\rm for~ all}~ u,v\in X$~
			$(u,v)\in \mathcal{R}\Rightarrow (fu,fv)\in \mathcal{R}.$
		\end{defn}
		
		\begin{defn}  \cite{Alamimdad2} \label{3.6} Let $(f,g)$ be a pair of
			self-mappings defined on a non-empty set $X$. Then a binary relation $\mathcal{R}$ on $X$  is called $(f,g)$-closed if  ${\rm for~ all}~ u,v\in X,$
				$(gu,gv)\in \mathcal{R}\Rightarrow (fu,fv)\in \mathcal{R}.$
		\end{defn}
		
		Notice that on setting $g=I,$ (the identity mapping on $X)$ Definition \ref{3.6} reduces to Definition \ref{3.5}.
		
		\begin{defn} \cite{Alamimdad} Let $\mathcal{R}$ be a binary relation defined on a non-empty set $X$. Then a sequence $\{u_n\}\subset X$ is said to be an $\mathcal{R}$-preserving if $(u_n,u_{n+1})\in\mathcal{R},\;\;\forall~n\in \mathbb{N}_{0}.$
		\end{defn}
		
		\begin{defn} \cite{Alamimdad2} Let $(X,d)$ be a metric space equipped with a binary relation $\mathcal{R}$.
			Then $(X,d)$ is said to be an $\mathcal{R}$-complete if every $\mathcal{R}$-preserving Cauchy sequence in $X$ converges to a point in $X$.
		\end{defn}
		
		\begin{rem} \cite{Alamimdad2} \label{rmk1} Every complete metric space is $\mathcal{R}$-complete, where $\mathcal{R}$ denotes a
			binary relation. Moreover, if $\mathcal{R}$ is universal relation, then notions of completeness
			and $\mathcal{R}$-completeness are same.
		\end{rem}
		
		\begin{defn} \cite{Alamimdad2}\label{3.9} Let $(X,d)$ be a metric space equipped with a binary relation $\mathcal{R}$. Then a mappings $f: X\rightarrow X$ is said to be an $\mathcal{R}$-continuous at $u$ if $u_n\stackrel{d}{\longrightarrow} u,$ for any $\mathcal{R}$-preserving sequence $\{u_n\}\subset X,$ we have $fu_n\stackrel{d}{\longrightarrow} fu$. Moreover, $f$ is said to be an
			$\mathcal{R}$-continuous if it is $\mathcal{R}$-continuous at every point of $X$.
		\end{defn}
		
		\begin{defn} \cite{Alamimdad2}\label{3.10} Let $(f,g)$ be a pair of self-mappings defined on a metric space $(X,d)$ equipped with a binary relation $\mathcal{R}$. Then $f$ is said to be a $(g,\mathcal{R})$-continuous at $x$ if $gu_n\stackrel{d}{\longrightarrow} gu$, for
			any $\mathcal{R}$-preserving sequence $\{u_n\}\subset X,$ we have
			$fu_n\stackrel{d}{\longrightarrow} fu$. Moreover, $f$ is called a
			$(g,\mathcal{R})$-continuous if it is $(g,\mathcal{R})$-continuous at every point of $X$.
		\end{defn}
		
		Notice that on setting $g=I$ (the identity mapping on $X)$, Definition \ref{3.10} reduces to Definition \ref{3.9}.
		
		\begin{rem} \label{rmk2} Every continuous mapping is $\mathcal{R}$-continuous, where $\mathcal{R}$ denotes a
			binary relation. Moreover, if $\mathcal{R}$ is universal relation, then notions of $\mathcal{R}$-continuity
			and continuity are same.
		\end{rem}
		
		\begin{defn} \cite{Alamimdad} \label{3.11}Let $(X,d)$ be a metric space. Then a binary relation $\mathcal{R}$ on $X$ is said to be
			$d$-self-closed if for any $\mathcal{R}$-preserving sequence
			$\{u_n\}$ with $u_n\stackrel{d}{\longrightarrow} u$, there
			is a subsequence $\{u_{n_k}\}{\rm \;of\;} \{u_n\}$
			such that $[u_{n_k},u]\in\mathcal{R},~~{\rm for~all}~k\in \mathbb{N}_{0}.$
		\end{defn}
		
			\begin{defn} \cite{Alamimdad2} \label{3.12}Let $g$ be a self-mapping on a metric space $(X,d)$. Then a binary relation $\mathcal{R}$ on $X$ is said to be
				$(g,d)$-self-closed if for any $\mathcal{R}$-preserving sequence
				$\{u_n\}$ with $u_n\stackrel{d}{\longrightarrow} u$, there
				is a subsequence $\{u_{n_k}\}{\rm \;of\;} \{u_n\}$
				such that $[gu_{n_k},gu]\in\mathcal{R},~~{\rm for~all}~k\in \mathbb{N}_{0}.$
			\end{defn}
			Notice that under the consideration $g=I$ (the identity mapping on $X)$, Definition \ref{3.12} turn out to be Definition \ref{3.11}.
		\begin{defn} \cite{SametT2012} \label{3.13} Let $(X,d)$ be a metric space endowed with a arbitrary binary relation $\mathcal{R}$. Then a subset $D$ of $X$
			is said to be an $\mathcal{R}$-directed if for every pair of points $u,v$ in $D$, there
			is $w$ in $X$ such that $(u,w)\in\mathcal{R}$ and
			$(v,w)\in\mathcal{R}$.
		\end{defn}
		
		\begin{defn} \cite{Alamimdad2} \label{3.14} Let $g$ be a self-mapping on a metric space $(X,d)$ endowed with a binary relation $\mathcal{R}$. Then a subset $D$ of $X$
			is said to be a $(g,\mathcal{R})$-directed if for every pair of points $u,v$ in $D$, there
			is $w$ in $X$ such that $(u,gw)\in\mathcal{R}$ and
			$(v,gw)\in\mathcal{R}.$
		\end{defn}
		 Notice that on setting $g=I$ (the identity mapping on $X)$, Definition \ref{3.14} turn out to be Definition \ref{3.13}.
		
		\begin{defn} \cite{Alamimdad2}
			Let $(f,g)$ be a pair of self-mappings defined on a metric space $(X,d)$ equipped with a binary relation $\mathcal{R}$. Then the pair $(f,g)$ is said to be an
			$\mathcal{R}$-compatible if $\lim\limits_{n\to \infty}d(g(fu_n),f(gu_n))=0$, whenever $\lim\limits_{n\to \infty}g(u_n)=\lim\limits_{n\to \infty}f(u_n)$, for any sequence $\{u_n\}\subset X$ such
			that $\{fu_n\}$ and $\{gu_n\}$ are $\mathcal{R}$-preserving.
		\end{defn}

		\begin{lem}\label{lm1} \cite{HRS2011}
			Let $g$ be a self-mapping defined on a non-empty set $X$. Then there exists a subset $Z\subseteq X$ with $g(Z)=g(X)$ and $g: Z\to X$ is one-one.
		\end{lem}
		\vspace{.2cm}
		For a given non-empty set $X$, together with a binary relation $\mathcal{R}$ on $X$ and a pair of self-mappings $(f,g)$ on $X,$ we use the following notations:
		\begin{itemize}
			\item $C(f,g)$: the collection of all coincidence points of $(f,g)$;
			\item $\mathcal{M}_{f}(gu,gv)$ := $max\big\{d(gu,gv), d(gu,fu), d(gv,fv), \frac{1}{2}[d(gu,fv)+d(gv,fu)]\big\}$; and
			\item $\mathcal{N}_{f}(gu,gv)$ := $max\big\{d(gu,gv), \frac{1}{2}[d(gu,fu)+d(gv,fv)], \frac{1}{2}[d(gu,fv)+d(gv,fu)]\big\}$.	
		\end{itemize}
		
		\vspace{.3cm}
		\begin{rem}\label{rem3}
			Observe that, $\mathcal{N}_{f}(gu,gv)\leq\mathcal{M}_{f}(gu,gv),~~{\rm for~all}~u,v\in X.$
		\end{rem}
		
		\vspace{.3cm}
		Let $\Phi$ be the family of all mappings $\varphi : [0,\infty) \to [0,\infty)$ satisfying the following properties:
		\begin{description}
			\item[$(\Phi_1)$] $\varphi$ is increasing;
			\item[$(\Phi_2)$] $\displaystyle\sum_{n=1}^{\infty}\varphi^{n}(t)<\infty$ for each $t > 0$, where $\varphi^{n}$ is the $n$-th iterate of $\varphi$.
		\end{description}
		
		\begin{lem}\label{lm2} \cite{SametT2012}
			Let $\varphi\in\Phi$. Then for all $s>0,$ we have $\varphi(s)<s.$
		\end{lem}
		\begin{prop}\label{p2}
			 Let $(f,g)$ be a pair of self-mappings defined on a metric space $(X,d)$ equipped with a binary relation $\mathcal{R}.$ and $\varphi\in\Phi $. Then the following conditions are equivalent:
			\begin{description}
				\item[(I)] $d(fu,fv)\leq\varphi(\mathcal{M}_{f}(gu,gv)) ~with ~(gu,gv)\in \mathcal{R};$
				\item[(II)] $d(fu,fv)\leq\varphi(\mathcal{M}_{f}(gu,gv)) ~with ~[gu,gv]\in \mathcal{R}.$
			\end{description}
		\end{prop}
		
		\begin{proof}
			The implication $(II)\Rightarrow (I)$ is straightforward.
			
			\noindent To show that $(I)\Rightarrow (II)$,
			choose $u,v\in X$ such that $[gu,gv]\in \mathcal{R}$. If $(gu,gv)\in \mathcal{R}$, then $(II)$ immediately follows from $(I)$. Otherwise, if $(gv,gu)\in \mathcal{R}$, then by $(I)$ and the symmetry of metric $d$, we obtained the conclusion.
		\end{proof}

		\vspace{.3cm}
		 For the sake of completeness, we state the following theorems:
	
		\begin{thm}\label{th1.1} \cite[Theorems 3.2]{Berzig2012}
			Let $(f,g)$ be a pair of self-mappings defined on a metric space $(X,d)$ equipped with a symmetric closure $\mathcal{S}:=\mathcal{R}\cup \mathcal{R}^{-1}$ of any binary relation $\mathcal{R}.$ Suppose the following conditions hold:
			\begin{enumerate}
				\item [$(a)$]$(X,d)$ is complete;
				\item [$(b)$]there exists $w_0\in X$ such that $(gw_0,fw_0)\in \mathcal{S};$
				\item [$(c)$]$\mathcal{S}$ is $(f,g)$ closed;
				\item [$(d)$]$(X,d,\mathcal{S})$ is regular;
				\item [$(e)$]there exists $\varphi\in \Phi$ such that  $d(fu,fv)\leq \varphi(\mathcal{N}_{f}(gu,gv))$ ~ for all $u,v \in X$ with $(gu,gv)\in \mathcal{S}.$
			\end{enumerate}
			Then $(f,g)$ has a unique coincidence point. Moreover, if $C(f,g)$  is $(g,\mathcal{S})$-directed and $(f,g)$ is weakly compatible, then $(f,g)$ has a unique common fixed point.
		\end{thm}

		\begin{thm}\label{th1.2} \cite[Theorem 2]{Alamimdad2}
			Let $(f,g)$ be a pair of self-mappings defined on a metric space $(X,d)$ equipped with a binary relation  $\mathcal{R}$ and $Y$ a subspace of $X.$ Assume that the following conditions hold:
			\begin{enumerate}
				\item [$(f)$] $(Y,d)$ is $\mathcal{R}$-complete subspace of $X$;
				\item [$(g)$] $f(X)\subseteq Y\cap g(X)$;
				\item [$(h)$] $\exists~w_0\in X$ such that $(gw_0,fw_0)\in \mathcal{R};$
				\item [$(i)$] $\mathcal{R}$ is $(f,g)$-closed;
				\item [$(j)$] there exists $\alpha\in [0,1)$ such that $d(fu,fv)\leq \alpha d(gu,gv)$ ~for all $u,v\in X\;\textrm{with}\; (gu,gv)\in \mathcal{R}$;				
				\item [$(k)$] \begin{enumerate}
					\item [$(k_{1})$] $Y\subseteq g(X);$
					\item [$(k_{2})$] either $f$ is $(g,\mathcal{R})$-continuous or $f$ and $g$ are continuous or $\mathcal{R}|_Y$ is $d$-self-closed;
				\end{enumerate}
			\end{enumerate}
			\hspace{.5cm}or, alternatively
			\begin{enumerate}
				\item [$(l)$] \begin{enumerate}
					\item [$(l_1)$] $(f,g)$ is $\mathcal{R}$-compatible;
					\item [$(l_2)$] $g$ is $\mathcal{R}$-continuous;
					\item [$(l_3)$] $f$ is $\mathcal{R}$-continuous or $\mathcal{R}$ is $(g,d)$-self-closed.
				\end{enumerate}
			\end{enumerate}
			Then $(f,g)$ has a coincidence point.
		\end{thm}
		\vspace{.3cm}
		Indeed, the main results of this paper are based on the following points:
		\begin{itemize}
			\item Theorem \ref{th1.1} is improved by replacing symmetric closure $\mathcal{S}$ of any binary relation with arbitrary binary relation $\mathcal{R}$,
			\item Theorems \ref{th1.1} (upto coincidence point) and \ref{th1.2} are unified by replacing more general contraction condition,
			\item Theorem \ref{th1.1} is generalized by replacing comparatively weaker notions namely $\mathcal{R}$-completeness of any subspace $Y\subseteq X$, with $fX\subseteq Y\cap gX$ rather than completeness of whole space $X$,
			\item Theorem \ref{th1.1} is improved by replacing $d$-self-closedness or $(g,d)$-self-closedness of $\mathcal{R}$ instead of regularity of the whole space,
			\item some examples are addopted to demonstrate the realized improvement in the results proved  in this article.
		\end{itemize}
		
		\section{Main Results}
		Now, we are equipped to prove our main result as follows:
		\begin{thm}\label{th1}
			Let $(f,g)$ be a pair of self-mappings defined on a metric space $(X,d)$ equipped with a binary relation  $\mathcal{R}$. Assume that the conditions $(f),(g),(h),(i)$ and $(k)$ $(or ~(l))$ together with the following condition holds:
			\begin{enumerate}
				\item [$(m)$] there exists $\varphi\in\Phi $ such that $d(fu,fv)\leq \varphi( \mathcal{M}_{f}(gu,gv))$ ~(for all $u,v\in X\;\textrm{with}\; (gu,gv)\in \mathcal{R}).$
			\end{enumerate}
			Then $(f,g)$ has a coincidence point.
		\end{thm}
		\begin{proof}
			Let $w_0\in X$ such that $(gw_0,fw_0)\in \mathcal{R}$. Construct a Picard Jungck sequence $\{gw_n\}$, with the initial point $w_{0},~ i.e.,$ \begin{equation}\label{1}
			gw_{n+1}=fw_{n},\; \;{\rm for~all}~ n\in \mathbb{N}_0.
			\end{equation}
			Also as $(gw_0,fw_0)\in \mathcal{R}$ and $\mathcal{R}$ is $(f,g)$-closed, we have
			$$(fw_0,fw_1),(fw_2,fw_3),\cdots ,(fw_n,fw_{n+1}),\cdots\in \mathcal{R}.$$
			Thus,
			$$(gw_n,gw_{n+1})\in \mathcal{R},\;\;{\rm for ~all}~n\in \mathbb{N}_0,\eqno(2)$$
			therefore $\{gw_n\}$ is $\mathcal{R}$-preserving.
			From condition $(m)$, we have (for all $n\in \mathbb{N}$)
			$$d(gw_n,gw_{n+1})=d(fw_{n-1},fw_{n})\leq \varphi( \mathcal{M}_{f}(gw_{n-1},gw_{n}))\eqno(3)$$
			where, $$\mathcal{M}_{f}(gw_{n-1},gw_{n})\leq max \big\{d(gw_{n-1},gw_{n}),d(gw_{n-1},fw_{n-1}),d(gw_n,fw_n), \frac{1}{2}\big[d(gw_n,fw_{n-1})+d(gw_{n-1},fw_n)\big]\big\}$$
			on using (1) and tringular inequality, we have (for all $n\in \mathbb{N}$)
			$$\mathcal{M}_{f}(gw_{n-1},gw_{n})\leq max \big\{d(gw_{n-1},gw_{n}), d(gw_n,gw_{n+1})\big\}.\eqno(4)$$
			On using (3), (4) and the property $(\Phi_1)$, we obtain (for all $n\in \mathbb{N}$)
			$$d(gw_n,gw_{n+1})\leq \varphi\big( max \big\{d(gw_{n-1},gw_{n}), d(gw_n,gw_{n+1})\big\}\big).\eqno(5)$$
			Now, we show that the sequence $\{gw_n\}$ is Cauchy in $(X,d)$. In case $gw_{n_{0}}=gw_{n_{0}+1}$ for some $n_0\in \mathbb{N}_0,$ then the result is follows. Otherwise, $gw_n \neq gw_{n+1}$ for all $n\in \mathbb{N}_0.$
			Suppose that $d(gw_{n_1-1},gw_{n_1})\leq d(gw_{n_1},gw_{n_1+1}),~ {\rm for ~some}~ {n_1}\in \mathbb{N}$. On using (5) and Lemma \ref{lm1}, we get
			$$d(gw_{n_1},gw_{{n_1}+1})\leq\varphi( d(gw_{n_1},gw_{{n_1}+1}) )< d(gw_{n_1},gw_{{n_1}+1}),$$
			which is a contradiction. Thus $d(gw_n,gw_{n+1})< d(gw_{n-1},gw_{n})$  (for all $n\in \mathbb{N}$), so that
			$$d(gw_{n},gw_{n+1})\leq\varphi( d(gw_{n-1},gw_{n})),~ {\rm for ~all}~ n\in \mathbb{N}.$$
			Employing induction on $n$ and the property $(\Phi_1)$, we get
			$$d(gw_{n},gw_{n+1})\leq\varphi^{n}( d(gw_0,gw_1)),~ {\rm for ~all}~ n\in \mathbb{N}_0.$$
			Now, for all $ m,n\in \mathbb{N}_0$ with $m\geq n,$ we have
			\begin{eqnarray*}
				\nonumber d(gw_{n},gw_{m})&\leq& d(gw_{n},gw_{n+1})+d(gw_{n+1},gw_{n+2})+\cdots+d(gw_{m-1},gw_{m})\\
				&\leq&\varphi^{n}(d(gw_0,gw_1))+\varphi^{n+1}(d(gw_0,gw_1))+\cdots+\varphi^{m-1}(d(gw_0,gw_1))\\
				&=& \sum\limits_{k=n}^{m-1} \varphi^{k}(d(gw_0,gw_1))\\
				&\leq& \sum\limits_{k\geq n} \varphi^{k}(d(gw_0,gw_1))\\
				&\rightarrow& 0\;{\rm as}\; n\rightarrow \infty.
			\end{eqnarray*}
			Therefore, $\{gw_n\}$ is $\mathcal{R}$-preserving Cauchy sequence in $X$.
			As $\{gw_n\}\subseteq g(X)$ and $\{gw_n\}\subseteq Y\subseteq g(X)$ (due to (1)  and $(k_1)$), therefore $\{gw_n\}$ is $\mathcal{R}$-preserving Cauchy sequence in $Y.$ Since $(Y,d)$ is $\mathcal{R}$-complete, there exists $y\in Y$ such that $gw_n\stackrel{d}{\longrightarrow} y$. As $Y\subseteq g(X),$ there exists $x\in X$ such that $$ \displaystyle\lim_{n\to\infty}gw_n=y=gx.\eqno(6)$$
			\vspace{.3cm}
			Since $f$ is $(g,\mathcal{R}$)-continuous, and on using (1) and (6), we have
			$$\displaystyle\lim_{n\to\infty}gw_{n+1}=\displaystyle\lim_{n\to\infty}fw_{n}=fx.$$
			Due to uniqueness of the limit, we have $fx=gx.$
			Hence $x$ is a coincidence point of $(f,g)$.
			
			Next, we assume that $f$ and $g$ are continuous. From Lemma \ref{lm1}, there exists a subset $Z\subseteq X$ such that $g(Z)=g(X)$ and $g: Z\rightarrow X$ is one-one. Now, define $h:g(Z)\rightarrow g(X)$ by $$h(g(z))=fz ~\forall~gz\in g(Z),~z\in Z.\eqno(7)$$
			Since $g$ is one-one and $f(X)\subseteq g(Z)$, $h$ is well defined. As $f$ and $g$ are continuous, so is $h.$ By utilizing the fact that $g(Z)=g(X)$
			and condition $(g)$ and $(k_1)$, we have $f(X)\subseteq g(Z)\cap Y$ and $Y\subseteq g(X)$ which guaranty that availability of a sequence $\{w_n\}\subset Z$
			satisfying (1). Take $x\in Z,$ on using (6), (7) and continuity of $h$, we get
			$$fx=h(gx)=h(\displaystyle\lim_{n\to\infty}gw_n)=\displaystyle\lim_{n\to\infty}h(gw_n)=\displaystyle\lim_{n\to\infty}fw_n=gx.$$\
			Hence $x$ is a coincidence point of $(f,g)$.
			
			Finally, if $\mathcal{R}|_Y$ is $d$-self-closed, then for any $\mathcal{R}$-preserving sequence
			$\{gw_n\}$ in $Y$ with $gw_n\stackrel{d}{\longrightarrow} gx$, there is a subsequence
			$\{gw_{n_k}\}{\rm \;of\;} \{gw_n\} \;{\rm
				such~that}\;\;[gw_{n_k},gx]\in\mathcal{R}|_Y\subseteq \mathcal{R},~{\rm for~all}~ k\in \mathbb{N} _0.$

			\vspace{.3cm}
			Set $\delta:=d(fx,gx)\geq0$. Suppose on contrary that $\delta>0.$ On using condition $(m)$,
			Proposition \ref{p2} and $[gw_{n_k},gx]\in\mathcal{R},$ for all $k\in \mathbb{N} _0,$ we have
			$$d(gw_{{n_k}+1},fx)=d(fw_{{n_k}},fx)\leq \varphi( \mathcal{M}_{f}(gw_{{n_k}},gx)),\eqno(8)$$
			where, $$\mathcal{M}_{f}(gw_{{n_k}},gx)= max\big\{d(gw_{{n_k}},gx), d(gw_{{n_k}},gw_{{n_k}+1}), d(gx,fx), \frac{1}{2}[d(gw_{{n_k}},fx) +d(gx,gw_{{n_k}+1})]\big\}.$$
			If $\mathcal{M}_{f}(gw_{{n_k}},gx)=d(gx,fx)=\delta,$ then (8) reduces to
			$$d(gw_{{n_k}+1},fx)\leq \varphi(\delta),$$
			which on making $k \to \infty$, gives arise
			$$\delta\leq \varphi(\delta),$$
			which is a contradiction.
			
			Otherwise, if $\mathcal{M}_{f}(gw_{{n_k}},gx)= max\big\{d(gw_{{n_k}},gx), d(gw_{{n_k}},gw_{{n_k}+1}),
			\frac{1}{2}[d(gw_{{n_k}},fx) +d(gx,gw_{{n_k}+1})]\big\},$ then due to the fact that $gw_n\stackrel{d}{\longrightarrow} gx$,
			there exists a positive integer $N=N(\delta)$ such that
			$$\mathcal{M}_{f}(gw_{{n_k}},gx)\leq\frac{4}{5}\delta, ~{\rm {for~ all}}~k\geq N.$$
			As $\varphi$ is increasing, we have
			$$\varphi(\mathcal{M}_{f}(gw_{{n_k}},gx))\leq\varphi(\frac{4}{5}\delta), ~{\rm {for~ all}}~k\geq N.\eqno(9)$$
			On using (8) and (9), we get
			$$d(gw_{{n_k}+1},fx)=d(fw_{{n_k}},fx)\leq\varphi(\frac{4}{5}\delta), ~{\rm {for~ all}}~k\geq N.$$
			Letting $k \to \infty$ and using Lemma \ref{lm2}, we get
			$$\delta\leq\varphi(\frac{4}{5}\delta)<\frac{4}{5}\delta<\delta,$$
			which is again a contradiction. Hence, $\delta =0$, so that $d(fx,gx)=\delta =0\Rightarrow fx=gx.$\\
			Thus, $x$ is a coincidence point of $(f,g)$.

			\vspace{.3cm}
			Alternatively, we suppose that $(l)$ holds. Firstly, we suppose that $f$ is $\mathcal{R}$-continuous. As $\{gw_n\}\subset f(X)\subseteq Y$ (in view (1)) we notice that $\{gw_n\}$ is $\mathcal{R}$-preserving Cauchy sequence in $Y.$ Since $Y$ is $\mathcal{R}$-complete, there exists $y\in Y$ such that
			$$\displaystyle\lim_{n\to\infty}gw_n=y ~ {\rm \;and \;} \displaystyle\lim_{n\to\infty}fw_n=y.\eqno(10)$$\\
			As $\{fw_n\}$ and $\{gw_n\}$ are $\mathcal{R}$-preserving sequence (due to (1) and (2)), utilizing the condition $(l_1)$ and (10), we obtain
			$$\displaystyle\lim_{n\to\infty}d(gfw_n,fgw_n)=0.\eqno(11)$$\\
			Using (2), (10), $(l_2)$ and due to $f$ is $\mathcal{R}$-continuous, we have $$\displaystyle\lim_{n\to\infty}g(fw_n)=g(\displaystyle\lim_{n\to\infty}fw_n)=gy,\eqno(12)$$\\
			and $$\displaystyle\lim_{n\to\infty}f(gw_n)=f(\displaystyle\lim_{n\to\infty}gw_n)=fy.\eqno(13)$$
			On using (11)--(13) and continuity of $d,$ we have $fy$=$gy.$ Hence $y$ is a coincidence point of $(f,g)$.

			\vspace{0.5mm}
			Lastly, assume that $\mathcal{R}$ is $(g,d)$-self-closed. As $\{gw_n\}$ is $\mathcal{R}$-preserving
		    and $gw_n\stackrel{d}{\longrightarrow}y $ (due to $(g,d)$-self-closedness of $\mathcal{R}$), there exists a subsequence $\{gw_{n_k}\}$ of $\{gw_n\}$
			such that $$[ggw_{n_k},gy]\in \mathcal{R}, ~\forall~ k\in \mathbb{N}_0.$$
			Since $gw_{n}\stackrel{d}{\longrightarrow}y $, therefore $gw_{n_k}\stackrel{d}{\longrightarrow}y $  for any subsequence $\{gw_{n_k}\}$  of $\{gw_{n}\}.$ \\
		
			Set $\eta:= d(fy,gy)\geq 0.$  Suppose on contrary that $\eta>0.$ On utilizing the condition $(m)$,
			Proposition \ref{p2} and $[ggw_{n_k},gy]\in\mathcal{R},$ for all $k\in \mathbb{N} _0,$ we have
			$$d(fgw_{{n_k}+1},fy)\leq \varphi( \mathcal{M}_{f}(ggw_{{n_k+1}},gy)),\eqno(14)$$
			\begin{multline*}
			{\rm  where},\:~ \mathcal{M}_{f}(ggw_{{n_k+1}},gy)= max\big\{d(ggw_{{n_k+1}},gy), d(ggw_{{n_k+1}},fgw_{{n_k}+1}),\\
			d(gy,fy), \frac{1}{2}[d(ggw_{{n_k+1}},fy) +d(gy,fgw_{{n_k}+1})]\big\}.
			\end{multline*}
			If $\mathcal{M}_{f}(ggw_{{n_k}},gy)=d(gy,fy)=\eta,$ then (14) yields
			$$|d(fgw_{{n_k}+1},gfw_{{n_k}+1})-d(gfw_{{n_k}+1},fy)|\leq d(fgw_{{n_k}+1},fy)\leq \varphi(\eta),$$
		on making $k \to \infty$; using (10), (11), continuity of $d$ and $\mathcal{R}$-continuity of $g$, we get
			$$\eta\leq \varphi(\eta),$$ which is a contradiction.
			Therefore, $\eta =0$, so that
			$$d(fy,gy)=\eta =0\Rightarrow fy=gy.$$
			Hence $y$ is a coincidence point of $(f,g)$.
			
            Otherwise, let $\mathcal{M}_{f}(ggw_{{n_k+1}},gy)= max\big\{d(ggw_{{n_k+1}},gy), d(ggw_{{n_k+1}},fgw_{{n_k}+1}), \frac{1}{2}[d(ggw_{{n_k+1}},fy) +d(gy,fgw_{{n_k}+1})]\big\}.$ Now, on using triangular inequality, we have
			\begin{multline*}
			\mathcal{M}_{f}(ggw_{{n_k+1}},gy)\leq max\big\{d(ggw_{{n_k+1}},gy), d(ggw_{{n_k+1}},ggw_{{n_k}+2})\\ +d(ggw_{{n_k}+2},fgw_{{n_k}+1}),
			\frac{1}{2}[d(ggw_{{n_k+1}},fy) +d(gy,gfw_{{n_k}+1})+d(gfw_{{n_k}+1},fgw_{{n_k}+1})]\big\},
			\end{multline*}
			On making $k \to \infty$, on using (1), (10), (11), continuity of $d$ and $\mathcal{R}$-continuity of $g$ , we get
			$$\displaystyle\lim_{n\to \infty} \mathcal{M}_{f}(ggw_{{n_k+1}},gy)=\frac{1}{2}\eta.$$
			Since $\eta >0$. By definition, there exists a positive integer $N=N(\eta)$ such that
			$$\mathcal{M}_{f}(ggw_{{n_k+1}},gy)\leq\frac{4}{5}\eta, ~{\rm {for~ all}}~k\geq N.$$
			As $\varphi$ is increasing, we have
			$$\varphi(\mathcal{M}_{f}(ggw_{{n_k+1}},gy))\leq\varphi(\frac{4}{5}\eta), ~{\rm {for~ all}}~k\geq N,$$
			again (14) yields that
			$$|d(fgw_{{n_k}+1},gfw_{{n_k}+1})-d(gfw_{{n_k}+1},fy)|\leq d(fgw_{{n_k}+1},fy)\leq\varphi(\frac{4}{5}\eta), ~{\rm {for~ all}}~k\geq N.\eqno(15)$$
			Hence,
			$$|d(fgw_{{n_k}+1},gfw_{{n_k}+1})-d(gfw_{{n_k}+1},fy)|\leq (\frac{4}{5}\eta), ~{\rm {for~ all}}~k\geq N.$$
			Letting $k \to \infty$, on using (10), (11), continuity of $d$ and $\mathcal{R}$-continuity of $g$, we get
			$$\eta\leq\varphi(\frac{4}{5}\eta)<\frac{4}{5}\eta<\eta,$$
			which is again a contradiction. Hence, $\eta =0$, so that
			$$d(fy,gy)=\eta =0\Rightarrow fy=gy.$$
			Hence, $y$ is a coincidence point of $(f,g)$. This completes the proof.
		\end{proof}
		\vspace{.3cm}
		On account  taking $Y=X$ in Theorem \ref{th1}, we deduce a corollary which is sharpened version of Theorem \ref{th1.1} up to coincidence point in view of comparatively weaker notions in the considerations of completeness, regularity and contraction condition.
		
		\begin{cor}\label{cor0}
			Let $(f,g)$ be a pair of self-mappings defined on a metric space $(X,d)$ equipped with a binary relation  $\mathcal{R}.$
			Suppose that the conditions $(h),(i),(m)$ together with the following conditions hold:
		
			\begin{enumerate}
			\item [$(n)$] $(X,d)$ is $\mathcal{R}$-complete;
			\item [$(o)$] $f(X)\subseteq g(X)$;				
			\item [$(p)$] $g$ is onto together the condition $(k_{2})$ ~[or, alternatively condition $(l)$].
	\end{enumerate}
		Then $(f,g)$ has a coincidence point.
		\end{cor}

	In liu of Remarks \ref{rem3}, Theorem \ref{th1} reduces to the following corollary.
		\begin{cor}\label{cor2}
			Let $(f,g)$ be a pair of self-mappings defined on a metric space $(X,d)$ endowed with a binary relation  $\mathcal{R}$ and $Y$ be a subspace of $X.$ Assume the conditions $(f),(g),(h),(i)$ and $(k)$ $(or~(l))$ together with the following condition holds:
			\begin{enumerate}
				\item [$(q)$] there exists $\varphi\in\Phi $ such that (for~all~$ u,v\in X\;\textrm{with}\; (gu,gv)\in \mathcal{R}$)
				$$ d(fu,fv)\leq \varphi\big(max\big\{d(gu,gv), \frac{1}{2}[d(gu,fu)+d(gv,fv)],
				\frac{1}{2}[d(gu,fv)+d(gv,fu)]\big).$$
			\end{enumerate}
			Then $(f,g)$ has a coincidence point.
		\end{cor}
		\vspace{.3cm}
		Now, we establish the following results for the uniqueness of common fixed point (corresponding to Corollary \ref{cor2}):
	
		\begin{thm}\label{th2}
			In additi
			on to the hypotheses of Corollary \ref{cor2}, suppose that the following condition holds:\\
			$(r): ~ f(X) ~is ~(g,\mathcal{R}^{s})-directed.$\\
			 Then $(f,g)$ has a unique point of coincidence. Moreover, if $(f,g)$ is weakly compatible, then $(f,g)$ has a unique common fixed point.
		\end{thm}
		
		\begin{proof} We prove the result in three steps.\\
			\noindent{\bf{Step 1:}}
			By Corollary \ref{cor2}, $C(f,g)$ is non-empty. If $C(f,g)$ is singleton, then there is nothing to prove. Otherwise, to substantiate the proof, take two arbitrary elements $u,v ~{\rm in}~ C(f,g),$ so that
			$$fu=gu=\overline{x}\;{\rm and}\;fv=gv=\overline{y}.$$
			Now, we are required to show that $\overline{x}=\overline{y}$.
			Since $C(f,g)\subseteq fX$ and $fX~{\rm is}~(g,\mathcal{R}^s)$-directed, there exists $u_0\in X$ such that $[\overline{x},gu_0]\in\mathcal{R}$ and $[\overline{y},gu_0]\in\mathcal{R}$. Now, we construct a sequence $\{gu_n\}$ corresponding to $u_0$, so that $gu_{n+1}=fu_{n} ~{\rm for~all}~ n\in\mathbb{N}_0.$
			
			\vspace{.3cm}
			We claim  that $\displaystyle\lim_{n\to\infty}d(\overline{x},gu_n)=0.$ If $d(\overline{x},gu_{n_0})=0,$ for some $n_0\in\mathbb{N}_0$, then there is nothing to prove. Otherwise, $d(\overline{x},gu_n)>0,$ for all $n\in\mathbb{N}_0.$
			\noindent As $[\overline{x}, gu_{n}]\in \mathcal{R}$, for all $n\in \mathbb{N}_0$ (due to the fact that $(f,g)$-closedness of $\mathcal{R}$ and $[\overline{x}, gu_0]\in \mathcal{R}$), by Proposition \ref{p2} and hypothesis $(q)$, we get
			$$d(\overline{x}, gu_{{n}+1})=d(fu, fu_{n})\leq \varphi( \mathcal{N}_{f}(gu, gu_{{n}})),\eqno(16)$$
			where, \begin{eqnarray*}
				\mathcal{N}_{f}(gu, gu_{{n}}) &=& max\big\{ d(gu,gu_{{n}}), \frac{1}{2}[d(gu,fu)+d(gu_{{n}},fu_{{n}})],
				\frac{1}{2}[d(gu,fu_{{n}}) +d(gu_{n},fu)\big]\big\} \\
				&\leq& max\big\{d(gu,gu_n),\frac{1}{2}[d(gu_{{n}},gu)+d(fu_{n},gu)], \frac{1}{2}[d(gu,fu_{{n}}) +d(gu_{n},fu)]\big\}\\
				&\leq& max\big\{d(gu,gu_n),\frac{1}{2}[d(gu_{{n}},gu)+d(fu_{n},gu)]\big\}\\
				&\leq& max\big\{d(gu,gu_n), d(gu,fu_{{n}})\big\}\\
				&=&max\big\{d(gu,gu_{{n}}), d(gu,gu_{{n}+1})\big\},
			\end{eqnarray*}
		    on using this and property $(\Phi_1)$ (16) yields (for all $n\in\mathbb{N}_0$)
			\begin{eqnarray*}
				d(gu, gu_{{n}+1}) &\leq& \varphi\big( max\big\{d(gu,gu_{{n}}), d(gu,gu_{{n}+1})\big\}\big)\\
				&= & \varphi\big(d(gu,gu_{{n}})\big),
			\end{eqnarray*}
			otherwise, we get a contradiction.
			So, by induction on $n$, we get
			$$d(gu,gu_{{n}})\leq\varphi^{n}(d(gu,gu_{{0}})), {\rm ~for~ all} ~n\in \mathbb{N}_0,$$
			which on making $n\to \infty $ and using the property $(\Phi_2)$, we get
			$$\displaystyle\lim_{n\to\infty}d(gu,gu_n)=0.\eqno(17)$$
			Similarly, we can obtain
			$$\displaystyle\lim_{n\to\infty}d(gv,gu_n)=0.\eqno(18)$$
			Using (17) and (18), we have
			\begin{eqnarray*}
				d(\overline{x},\overline{y}) &\leq& d(gu,gu_n)+d(gu_n, gv)  \\
				&\to& 0,~ \text{as} ~n\to \infty\\
				\Rightarrow \overline{x} &=& \overline{y} ~~i.e.,~ (f,g) ~\; {\rm has~ a~ unique~point~of~coincidence}\;.
			\end{eqnarray*}
			
			\noindent{\bf{Step 2:}} Now, we claim that the pair  $(f,g)$ has a common fixed point, let $x\in C(f,g), ~i.e., ~fx=gx$. Due to weakly compatibility of the pair $(f,g)$, we have
			$$f(gx)=g(fx)=g(gx).\eqno(19)$$
			Put $gx=y$. Then from (19), $fy=gy.$ Hence $y$ is also a coincidence point of $f$ and $g$.
			In view of Step 1, we have $$y=gx=gy=fy,$$ so that $y$ is a common fixed point $(f,g)$.\\
			
			\noindent{\bf{Step 3:}} To prove the uniqueness of common fixed point of $(f,g)$, let us assume that $w$ is another common fixed point of $(f,g)$.
			Then $w\in C(f,g),$ by Step 1, $$w=gw=gy=y.$$
			Hence $(f,g)$ has a unique common fixed point.
		\end{proof}
		
		\vspace{.3cm}
		\begin{rem}\label{rem4}
			In view of Theorem \ref{th2}, we have used comparatively more natural condition `` $(g,\mathcal{R}^s)$-directedness of $f(X)$" instead of `` $(g,\mathcal{R}^{s})$-directedness of $C(f,g)$" which is too restrictive. Our proof carry on even if we take ``$C(f,g)~is~ (g,\mathcal{R}^{s})$-directed". Since point of coincidence implies that coincidence point due to weakly compatible of $(f,g)$, as in our Theorem \ref{th2} we want to find unique common fixed point of $f$ and $g$ which is the point in $C(f,g).$
		\end{rem}
		\begin{thm}\label{th3}
			In addition to the hypotheses of Theorem \ref{th1}, assume the condition $(r)$ together with the following condition holds:
			$$(s):  one ~of ~f ~and ~g ~is ~one ~to~one~ .$$ Then $(f,g)$ has a unique coincidence point.
		\end{thm}
		\begin{proof}
			Due to Theorem \ref{th1}, $C(f,g)\ne \emptyset.$ Let $u,v\in C(f,g)$, and hence in similar lines of the proof of Theorem \ref{th2}, we have
			$$gu=fu=fv=gv.$$
			Since either $f$ or $g$ is one-one, we have $$u=v.$$
		\end{proof}
		
		Notice that Theorem \ref{th3} is a natural improved version of Theorem 4 due to Alam and Imdad \cite{Alamimdad2}. 	
		
		\vspace{.3cm}
		\begin{thm}\label{th4} In addition to the hypotheses of Theorem \ref{th1}, assume the following condition holds:
			\begin{enumerate}
				\item [$(t)$] $\mathcal{R}|_{fX}~is~complete.$
			\end{enumerate}
			Then $(f,g)$ has a unique point of coincidence. Moreover, if $(f,g)$ is weakly compatible, then $(f,g)$ has a unique common fixed point.
		\end{thm}
		\begin{proof}
			From Theorem \ref{th1}, we have $C(f,g)\ne\emptyset$. If $C(f,g)$ is singleton, then proof is over. Otherwise, choose any two elements $x\ne y ~{\rm in}~ C(f,g),$ so that
			$$fx=gx=\overline{x}\;{\rm and}\;fy=gy=\overline{y}.$$
			As $\mathcal{R}|_{fX}$ is complete, $[\overline{x}, \overline{y}]\in \mathcal{R}$. Using Proposition \ref{p2} and condition $(m)$, we get
			\begin{eqnarray*}
			d(\overline{x},\overline{y})=d(fx,fy) &\leq& \varphi\big(max\big\{d(gx,gy),
			d(gx,fx),d(gy,fy), \frac{1}{2}[d(gx,fy)+d(gy,fx)]\big\}\big)\\
				&=& \varphi\big(d(gx,gy)\big)\\
				&<& d(gx,gy)= d(\overline{x},\overline{y}),
			\end{eqnarray*}
			which is a contradiction, hence $d(\overline{x},\overline{y})=0$, therefore $\overline{x}=\overline{y}.$
			Thus $(f,g)$ has a unique point of coincidence. Thus the remaining part of the proof can be obtained from Theorem \ref{th2}.
		\end{proof}
		\vspace{.3cm}
		\begin{rem}
			Indeed, Theorem \ref{th4} is more general as compared to Corollary 5.1 of Berzig \cite{Berzig2012} and Corollary 3 due to Alam and Imdad \cite{Alamimdad2}.
		\end{rem}

		\vspace{.3cm}
		In regard of Remark \ref{rem4}, on considering symmetric closure $\mathcal{S}$ of any binary relation $\mathcal{R}$  in Theorem \ref{th2}, we obtain the following sharpened version of Theorem \ref{th1.2}.
		
		\begin{cor}\label{cor3}
			Let $(f,g)$ be a pair of self-mappings defined on a metric space $(X,d)$ endowed with symmetric closure $\mathcal{S}$ of any arbitrary binary relation defined on $X$ and $Y$ be a subspace of $X$.
			Assume that the conditions $(b),(c),(e)$ and $(g)$ together with the following conditions hold:
			\begin{enumerate}
				\item [$(u)$] $(Y,d)$ is $\mathcal{S}$-complete subspace of $X$;
				\item [$(v)$] \begin{enumerate}
					\item [$(v_{1})$] $Y\subseteq g(X);$
					\item [$(v_{2})$] either $f$ is $(g,\mathcal{S})$-continuous or $f$ and $g$ are continuous or $\mathcal{S}|_Y$ is $d$-self-closed;
				\end{enumerate}
			\end{enumerate}
			\hspace{.5cm}or, alternatively
			\begin{enumerate}
				\item [$(w)$]  \begin{enumerate}
					\item [$(w_1)$] $(f,g)$ is $\mathcal{S}$-compatible;
					\item [$(w_2)$] $g$ is $\mathcal{S}$-continuous;
					\item [$(w_3)$] either $f$ is $\mathcal{S}$-continuous or $\mathcal{S}$ is $(g,d)$-self-closed.
				\end{enumerate}
			\end{enumerate}
			Then $(f,g)$ has a coincidence point. Moreover, if $C(f,g)$ is $(g,\mathcal{R}^{s})$-directed and $(f,g)$ is weakly compatible,
			then $(f,g)$ has a unique common fixed point.
		\end{cor}
		
		Notice that the hypotheses `$\mathcal{S}$ is $(f,g)$-closed' is equivalent to $f$ is a  '$g$-comparative' and `$\mathcal{S}|_Y$ is $d$-self-closed'  is more natural `the regular property of $(Y,d,\mathcal{S})$'. Further `$\mathcal{S}$ is $(g,d)$-self-closed' is more natural the `$\mathcal{S}$ is $d$-self-closed'.

		\section{Consequences}
		As consequences of our former proved results, we deduce several well known results of the existing literature.

	\vspace{.3cm}
		On the setting $\varphi(t)=kt$, with $k\in [0,1)$, we obtain the following corollaries which are immediate consequences of Theorem \ref{th2}.
		
		\begin{cor}\label{cor5}
			Let $(f,g)$ be a pair of self-mappings defined on a metric space $(X,d)$ equipped with a binary relation  $\mathcal{R}$ and  $Y$ a subspace of $X.$
			Suppose that the conditions $(f),(g),(h),(i)$ and $(k)$ $(or~(l))$ together with  the following condition holds:
			
			\begin{enumerate}
				\item [$(q_{1})$] there exists $k\in [0,1]$ such that $(~for~ all~u,v\in X\;\textrm{with}\; (gu,gv)\in \mathcal{R})$
				
				$$d(fu,fv)\leq k(max\big\{d(gu,gv), \frac{1}{2}[d(gu,fu)+d(gv,fv)],\frac{1}{2}[d(gu,fv)+d(gv,fu)]\big\}\big).$$
			\end{enumerate}
			Then $(f,g)$ has a coincidence point. Moreover, if $f(X)$ is $(g,\mathcal{R}^{s})$-directed and  $(f,g)$ is weakly compatible,
			then $(f,g)$ has a unique common fixed point.
		\end{cor}
		\begin{rem}
			Corollary \ref{cor5} is a sharpened version of Corollary 5.10 of Berzig \cite{Berzig2012} and Corollary 3.3 (corresponding to condition (20)) due to Ahmadullah et al. \cite{AhmadIA}.
		\end{rem}
		\begin{cor}\label{cor6}
			Let $(f,g)$ be a pair of self-mappings defined on a metric space $(X,d)$ equipped with a binary relation  $\mathcal{R}.$ Suppose that the conditions $(f),(g),(h),(i)$ and $(k)$ $(or~(l))$ together with the following condition holds:
			\begin{enumerate}
				\item [$(q_{2})$] there exist $a,~b,~ c\geq 0 ~with~a+2b+2c<1$ such that (for~ all~ $u,v\in X\;\textrm{with}\; (gu,gv)\in \mathcal{R}$)
				$$d(fu,fv)\leq ad(gu,gv)+ b[d(gu,fu)+d(gv,fv)]+c[d(gu,fv)+d(gv,fu)].$$
			\end{enumerate}
			Then  $(f,g)$ has a coincidence point. Moreover, if $f(X)$ is $(g,\mathcal{R}^{s})$-directed and $(f,g)$ is weakly compatible,
			then  $(f,g)$ has a unique common fixed point.
		\end{cor}
		
		\begin{rem}
			Corollary \ref{cor6} remains a sharpened version of Corollary 5.11 due to Berzig \cite{Berzig2012} and Corollary 3.3, (corresponding to condition (22)) in view of Ahmadullah et al. \cite{AhmadIA}.
		\end{rem}
		\begin{rem}
			If $b=0$ and $c=0$ in Corollary \ref{cor6}, then we deduces Theorem \ref{th1.2} (see Alam and Imdad \cite{Alamimdad2}).
		\end{rem}

		\begin{cor}\label{cor8}
			Let $(f,g)$ be a pair of self-mappings defined on a metric space $(X,d)$ equipped a binary relation  $\mathcal{R}$ and $Y$ a subspace of $X$. Assume that the conditions  $(f),(g),(h),(i)$ and $(k)$  (or $(l)$) together with the following condition holds:
			\begin{enumerate}
				\item [$(q_{3})$]there exists $k\in [0,1/2)$ such that (for all $u,v\in X\;\textrm{with}\; (gu,gv)\in \mathcal{R}$)\\
				$$d(fu,fv)\leq k[d(gu,fu)+d(gv,fv)].$$
			\end{enumerate}
			Then $(f,g)$ has a coincidence point. Moreover, if $f(X)$ is $(g,\mathcal{R}^{s})$-directed and  $(f,g)$ is weakly compatible,
			then $(f,g)$ has a unique common fixed point.
		\end{cor}
		\begin{rem}
			Corollary \ref{cor8} remains a improved version of Corollary 5.13 established in Berzig \cite{Berzig2012} and Corollary 3.3 (corresponding to condition (18)) in Ahmadullah et al. \cite{AhmadIA}.
		\end{rem}
		\begin{cor}\label{cor9}
			Let $(f,g)$ be a pair of self-mappings defined on a metric space $(X,d)$ equipped with a binary relation  $\mathcal{R}$ and $Y$ a subspace of $X.$ Assume that the conditions  $(f),(g),(h),(i)$ and $(k)$ (or $(l)$) togetrher with the following condition holds:
			\begin{enumerate}
				\item [$(q_{4})$]there exists $k\in [0,1/2)$ such that (for all $u,v\in X\;\textrm{with}\; (gu,gv)\in \mathcal{R}$)\\
			$$d(fu,fv)\leq k\big[d(gu,fv)+d(gv,fu)\big].$$
			\end{enumerate}
			Then $(f,g)$ has a coincidence point. Moreover, if $f(X)$ is $(g,\mathcal{R}^{s})$-directed and  $(f,g)$ is weakly compatible,
			then $(f,g)$ has a unique common fixed point.
		\end{cor}
		\begin{rem}
			Corollary \ref{cor9} is an improved version of Corollary 5.14 of Berzig \cite{Berzig2012} and Corollary 3.3 (corresponding to condition (19)) due to Ahmadullah et al. \cite{AhmadIA}.
		\end{rem}

		\begin{rem}
			Under the consideration $g=I$ (identity mapping on $X$), Theorems \ref{th1} and \ref{th2} deduce the fixed point results of Ahmadullah et al. \cite[Theorem 2.1 and 2.5]{AhmadIG}.
		\end{rem}
		
		\begin{rem}
			On setting $g=I$ (identity mapping on $X$), in Corollaries \ref{cor0}-\ref{cor9}, we deduce the fixed point results which are the sharpened version of several results in the existing literature.
		\end{rem}
	Also under the universal relation $(i.e.,~\mathcal{R}=X\times X),$ Theorems \ref{th2} and \ref{th4} unify to the following lone corollary:
		
		\begin{cor}\label{cor10}
			Let $(X,d)$ be a metric space and $(f,g)$ a pair of self-mappings on
			$X$. Suppose that the following conditions hold:
			\begin{enumerate}
				\item [$(A)$]there exists $Y\subseteq X,$ $f(X)\subseteq Y\subseteq g(X)$ such that $(Y,d)$ is complete;
				\item [$(B)$] there exists $\varphi\in\Phi $ such that (for all $u,v\in X$)
				$$d(fu,fv)\leq \varphi\big(max\big\{d(gu,gv), \frac{1}{2}[d(gu,fu)+d(gv,fv)], \frac{1}{2}[d(gu,fv)+d(gv,fu)]\big\}\big).$$
			\end{enumerate}
			Then $(f,g)$ has a unique common fixed point.
		\end{cor}
	
		\section{Illustrative Examples}
		In this section, we furnish some examples to demonstrate the realized improvement of our proved results.
		
		\begin{example}\label{exm1}
			Let $(X,d)$ be a metric space, where $X=(-2,4)$ and $d(x,y)=|x-y|, ~\forall~x~y \in X$.
			Now, define a binary relation $\mathcal{R}=\{(x,y)\in X^2~|~ x\geq y, x,y \geq 0\}\cup \{(x,y)\in X^2~|~ x\leq y, x,y \leq 0\}$, an increasing mapping $\varphi : [0,\infty)\to [0,\infty)$ by $\varphi(s)= \frac{1}{3}s$ and two self-mappings $f,g:X\to X$ by
			$$f(x)= 0, ~~\forall~ x \in (-2,4)
			~ \;{\rm and} \quad g(x)=\left\{
			\begin{array}{ll}
			\frac{x}{3}, & \hbox{$x\in (-2,3]$;} \\
			1, & \hbox{$x\in (3,4)$.}
			\end{array}
			\right.$$
			Let $Y=[-\frac{1}{2},1)$, so that $f(X)=\{0\}\subset Y\subset gX=(-\frac{2}{3},1]$ and $Y$ is $\mathcal{R}$-complete but $X$ is not $\mathcal{R}$-complete.
			Indeed, $\mathcal{R}$ is $(f,g)$-closed and $f$ and $g$ are $\mathcal{R}$-continuous.
			By straightforward calculations, one can easily verify hypothesis $(m)$ of Theorems \ref{th1}
			thus in all by Theorem \ref{th1} we obtain, $(f,g)$ has a coincidence point (Observe that, $C(f,g)=\{0\}$).
			Moreover, as $fX ~is~(g,\mathcal{R}^s)\text{-}$directed, $\mathcal{R}|_{fX}$ is complete and $(f,g)$ commute at their coincidence point $i.e., x=0$ therefore, all the hypotheses of Theorems \ref{th2} and \ref{th4} are satisfied, ensuring the uniqueness of the common fixed point. Notice that, $x=0$ is the only common fixed point of  $(f,g)$.
			
			\vspace{.3cm}
			With a view to show the genuineness of our results, notice that $\mathcal{R}$ is not symmetric and $\mathcal{R}$ can not  be a symmetric closure of any binary relation. Also $(X,d)$ is not complete and even not $\mathcal{R}$-complete which shows that Theorems \ref{th1}, \ref{th2} and \ref{th4} are applicable to the present example, while Theorem \ref{th1.1} and even Corollary \ref{cor0} are not, which substantiates the utility of Theorems \ref{th1}, \ref{th2} and \ref{th4}.
		\end{example}
		
		\vspace{.3cm}
		\begin{example}\label{exm2}
			Let $X=[0,4)$ with usual metric $d$ and $\mathcal{R}= \{(0,0),(0,1),(1,0),(1,1),(1,2),(2,3)\}$ be a binary relation whose symmetric closure $\mathcal{S}=\{(0,0),(0,1),(1,0),(2,3),(1,1),(1,2),(2,1),(3,2)\}$ and $(f,g)$ a pair of self-mappings on $X$ defined by
			$$f(x)=\left\{
			\begin{array}{ll}
			0, & \hbox{$x\in [0,4)\cap \mathbb{Q}$;} \\
			1, & \hbox{$x\in [0,4)\cap \mathbb{Q}^{c}$,}
			\end{array}
			\right.
			\;{\rm and}
			\quad g(x)=\left\{
			\begin{array}{ll}
			x, & \hbox{$x\in \{0,1,2\}$;} \\
			3, & \hbox{$ \;{\rm otherwise}.$}
			\end{array}
			\right.$$\\
			Let $Y=\{0,1\}$ which is $\mathcal{R}$-complete and $fX=\{0,1\}\subseteq Y\subseteq gX=\{0,1,2,3\}.$
			Define an increasing function $\varphi : [0,\infty)\to [0,\infty)$ by $\varphi(s)= \frac{5}{6}s.$\
			Clearly, $\varphi \in \Phi$ and both $f$ and $g$ are not continuous. Also, $\mathcal{R}$ is $(f,g)$-closed. Take any $\mathcal{R}$-preserving sequence
			$\{x_n\}$ in $Y,$ $i.e.,$  $$(x_n,x_{n+1})\in\mathcal{R}|_Y, ~{\rm for~ all}~ n\in \mathbb{N}
			~ {\rm with} ~x_n\stackrel{d}{\longrightarrow} x.$$
			Here, one can notice that if $(x_n,x_{n+1})\in\mathcal{R}|_{Y}$, for all $n\in\mathbb{N}$
			then there exists $N \in\mathbb{N}$
			such that $x_n=x \in \{0,1\}, ~{\rm for~ all}~ n\geq N$. So, we choose a subsequence
			$\{x_{n_k}\}$ of the sequence $\{x_n\}$ such that $x_{n_k}=x$, for all $k\in \mathbb{N}$, which
			amounts to saying that $[x_{n_k},x]\in \mathcal{R}|_{Y}$, for all $k\in \mathbb{N}$. Therefore,
			$\mathcal{R}|_{Y}$ is $d$-self-closed.
			
			Now, to substantiate the contraction condition $(m)$ of Theorems \ref{th1}. For this, we need to verify for $(gx,gy)\in \{(2,3)\},$
			otherwise, $d(fx,fy)=0.$ If $(gx,gy)\in \{(2,3)\}$ $\Rightarrow$ $x=2,~y\in [0,4)-\{0,1,2\}$,
			then there are two cases arises:\\
			{\bf Case (1):} If $x=2,~y\in ([0,4)-\{0,1,2\})\cap \mathbb{Q}$, then condition $(m)$ is obvious.\\
			{\bf Case (2):} If $x=2,~y\in ([0,4)-\{0,1,2,3\})\cap \mathbb{Q}^{c}$, then we have
			\begin{eqnarray*}
				d(f2,fy)=1&\leq & \varphi(max\{d(g2,gy), d(g2,f2),d(gy,fy),\frac{1}{2}[d(g2,fy)+d(gy,f2)]\})\\
				&=&  \varphi(max\{d(2,3), d(2,0),d(3,1),\frac{1}{2}[d(2,1)+d(3,0)]\}\\
				&=& \varphi(2).
			\end{eqnarray*}
		
			Thus all the conditions of Theorem \ref{th1}  are satisfied, hence $(f,g)$ has a coincidence point (namely $C(f,g)=\{0\}$). Also ${fX}$ is $(g,\mathcal{R}^{s})$-directed,  $(f,g)$ commutes at their coincidence point $i.e.,$ at $x=0$ and condition $(m)$ of Theorem \ref{th2} holds. Therefore all the hypotheses of Theorem \ref{th2} are satisfied. Notice that, $x=0$ is the only common fixed point of  $(f,g)$.
			
			Now, since $(gx,gy)=(2,3)\in \mathcal{R},$ clearly $x=2$, we choose $y=\sqrt{2}$ but $$1=d(f2,f\sqrt{2})\leq \alpha d(g2,g\sqrt{2})=\alpha,$$ which shows that contraction condition of Theorem \ref{th1.2} (due to Alam and Imdad \cite{Alamimdad2}) is not satisfied.
			Further, Theorem \ref{th1.1} is not applicable to the present example as underlying metric space $(X,d)$ is not complete and $\mathcal{R}$ is not symmetric closure of any binary relation. Thus, our results are an improvement over Theorem \ref{th1.1} (due to  Berzig \cite{Berzig2012}) and Theorem \ref{th1.2} (Alam and Imdad \cite{Alamimdad2}).
		\end{example}

		\vspace{.3cm}\noindent
		{\bf Acknowledgements:} All the authors are read and approved the final version of this paper.

	\end{document}